\RequirePackage{fix-cm}
\documentclass[11pt,a4paper,reqno]{article}
\usepackage{graphicx}
\usepackage{mathptmx}   
\usepackage{verbatim}

\usepackage[latin1]{inputenc}  
\usepackage{amsfonts}       
\usepackage{amsmath}
\usepackage{amssymb}
\usepackage{multirow}
\usepackage{graphicx}  
\usepackage{enumerate}
\usepackage{bibentry}
\usepackage{enumerate}
\usepackage{color}
\usepackage{amsmath}

\numberwithin{equation}{section}

\newcommand{\I}{\mathbb{I}}

\newcommand{\R}{\mathbb{R}}

\DeclareMathOperator{\esssup}{ess\,sup}
\DeclareMathOperator{\essinf}{ess\,inf}

\usepackage{amsthm}
\newtheorem{theorem}{Theorem}
\newtheorem{corollary}{Corollary}

\newtheorem{lemma}{Lemma}

\theoremstyle{definition}
\newtheorem{example}{Example}

\author{Jaakko Lehtomaa\footnote{e-mail address: jaakko.lehtomaa@helsinki.fi}}

\title{Asymptotic behaviour of ruin probabilities in a general discrete risk model using moment indices}

\begin{document}

\maketitle

\begin{abstract}
We study the rough asymptotic behaviour of a general economic risk model in a discrete setting. Both financial and insurance risks are taken into account. Loss during the first $n$ years is modelled as a random variable $B_1+A_1B_2+\ldots+A_1\ldots A_{n-1}B_n$, where $A_i$ corresponds to the financial risk of the year $i$ and $B_i$ represents the insurance risk respectively. Risks of the same year $i$ are not assumed to be independent.

The main result shows that ruin probabilities exhibit power law decay under general assumptions. Our objective is to give a complete characterisation of the relevant quantities that describe the speed at which the ruin probability vanishes as the amount of initial capital grows. These quantities can be expressed as maximal moments, called moment indices, of suitable random variables. In addition to the study of ultimate ruin, the case of finite time interval ruin is considered. Both of these investigations make extensive use of the new properties of moment indices developed during the first half of the paper.
\end{abstract}


\section{Introduction}
\label{intro}

\subsection{Motivation}\label{notmot}
Assume $(\Omega,\mathcal{F},P)$ is a probability space. Define a stochastic process $(Y_n)_{n=1}^\infty=(Y_n)$ on $(\Omega,\mathcal{F},P)$ by the formula
\begin{equation}\label{modeldesc}
Y_n=\sum_{i=1}^n A_1 \ldots A_{i-1} B_i,
\end{equation}
where the product over an empty set is understood to be $1$. The process \eqref{modeldesc} allows us to define two maximal random variables. We define
\begin{equation}\label{ultimatevariable}
\bar{Y}=\sup_{k \in \mathbb{N}} \{Y_k\}
\end{equation}
for an unbounded time interval and 
\begin{equation}\label{finitevariable}
\bar{Y}_n=\sup_{1\leq k \leq n} \{Y_k\}
\end{equation}
for a bounded time interval. Under additional assumptions we may also define perpetuity
\begin{equation}\label{infinitesumvariable}
Y_\infty=\sum_{i=1}^\infty A_1 \ldots A_{i-1} B_i.
\end{equation}

Processes of the form \eqref{modeldesc} have attracted a fair amount of interest in recent years. An intriguing article \cite{alsmeyer2009distributional} investigated this kind of process and recovered results concerning the maximal moments of the random variable $|Y_\infty|$. To name a few others, \cite{Tang2003299}, \cite{randomsum} and \cite{tang2004finite} have also used the model \eqref{modeldesc}, but mainly to study the bounded interval case. For example, \cite{Tang2003299} gives results concerning the moments of \eqref{finitevariable} under independence assumption. 

The contribution of the present paper is twofold. Firstly, we want to expand the study of moments of absolute values to the study of moments of, say, the positive parts of random variables. This allows us to study one-sided level crossing situations that are typically encountered in insurance mathematics in the form of ruin problems. Our proofs do not need complicated mathematical machinery, which allows us to work with truly minimal assumptions. Secondly, the dependence of the financial and insurance risks of the same year has not, to our knowledge, been investigated in the context of determining the maximal moments of processes \eqref{ultimatevariable} and \eqref{finitevariable}. Yet, such dependence can be found in many applications.

To give an example, we recall the famous ARCH(1) process $(X_n)$ with parameters $\beta,\lambda>0$. It is defined recursively by the formula 
\begin{equation*}
X_n=\sqrt{\beta+\lambda X_{n-1}^2}Z_n, \ n\in \mathbb{N},
\end{equation*}
where $(Z_n)$ is a sequence of independent and identically distributed, abbreviated IID, random variables and $X_0$ is independent of $(Z_n)$. For the squared process $(X_n^2)$ we have
\begin{equation}\label{arch1}
X_n^2=\beta Z_n^2+\lambda Z_n^2X_{n-1}^2.
\end{equation}
If $X_0=0$, iteration of \eqref{arch1} yields a random variable whose distribution is that of \eqref{modeldesc} with  $(A_i,B_i)=(\lambda Z_i^2,\beta Z_i^2)$. Clearly, $A_i$ and $B_i$ are not independent in this case. In Section \ref{stochdiff} below the connection between \eqref{modeldesc} and recursive equations of the type \eqref{arch1} is explained in detail.

Another natural example of the occurrence of \eqref{modeldesc} with dependent risks can be found from the continuous time analogue of \eqref{modeldesc}. Using the notation of Section 3 of \cite{Harri2001265}, the continuous time analogue $(Y_t^c)_{t\geq 0}$ of the process $(Y_n)$ is defined by 
\begin{equation}\label{contanalogue}
Y_t^c=\int_0^t \! A_{s-}^c \, \mathrm{d} B_s^c.
\end{equation}
In \eqref{contanalogue} the process $((\log A_t^c,B_t^c))_{t\geq 0}$ is defined on a suitable filtered probability space $(\Omega, \mathcal{F}_{t\geq 0},\mathbb{P})$. It is assumed to be stationary with independent increments and càdlàg sample paths. With choices
\begin{equation*}
A_t^c=\prod_{m=1}^{\lfloor t \rfloor} A_m 
\end{equation*}
and
\begin{equation*}
 B_t^c=\sum_{m=1}^{\lfloor t \rfloor} B_m
\end{equation*}
model \eqref{modeldesc} can be recovered from \eqref{contanalogue}. By the terminology of \cite{Harri2001265}, \emph{the associated discrete time process} of \eqref{contanalogue} is defined by fixing the IID random vectors $(A_1,B_1),(A_2,B_2),\ldots$ via equation
\begin{equation}
(A,B)\stackrel{d}{=}(A_1^c,Y_1^c),
\end{equation}
where $\stackrel{d}{=}$ denotes equality in distribution. The interesting observation is that the vector $(A,B)$ does not necessarily have independent components even if the original processes $(A_t^c)$ and $(B_t^c)$ are independent. Further background to processes of the type \eqref{contanalogue} and ruin analysis in these models can be found from  \cite{paulsen2002cramer}. The reader is also advised to see a survey article \cite{paulsen1998ruin} from the same author.

\subsection{Assumptions}\label{alkuassumptions}
The process $(Y_n)$ defined in \eqref{modeldesc} describes the total loss of an insurer at the end of the year $n\in \mathbb{N}=\{1,2,3,\ldots \}$. We assume that the sequence $(A_1,B_1),(A_2,B_2),\ldots$ satisfies the following assumptions.
\begin{enumerate}[I]
\item The sequence $((A_i,B_i))$ consists of IID random vectors.\label{as3}
\item The members of $(A_i)$ are strictly positive and $A$ is not the constant $1$. \label{as1}
\item The members of $(B_i)$ are real valued and $P(B>0)>0$.\label{as2}
\end{enumerate}

In a sense, assumptions \ref{as1} and \ref{as2} are minimal. If $A$ was the constant $1$, model \eqref{modeldesc} would reduce to a classical random walk. However, random walks have entirely different asymptotic behaviour. This is the reason they are only mentioned as comments in Section \ref{conclusions}. On the other hand, if the equality $P(B>0)=0$ was valid, ruin would be impossible.

\subsection{Background}

Assume that the insurer operates with initial capital $U_0>0$. Denote by $U_n$ the capital of the insurer at the end of the year $n$. Let $r_n$ be the associated stochastic return on investment over the period from $n-1$ to $n$. Our assumption is that the sequence $(U_n)$ satisfies the recursive formula
\begin{equation}\label{modelrec}
U_n=(1+r_n)(U_{n-1}-B_n).
\end{equation} 
For further background of this recursion and the relation between \eqref{modeldesc} and \eqref{modelrec}, see \cite{Harri1999319}, \cite{Harri2001265} and especially \cite{nyrhinen2011stochastic}.

The random variable $A$ of model \eqref{modeldesc} can be regarded as a stochastic discounting factor $1/(1+r)$. Since $A>0$, the random variable $r$ must satisfy the inequality $r>-1$. This means that losses on financial markets are possible, but one cannot lose more than what is initially invested.

We define the time of ruin $T_{U_0}=T$ by
\begin{equation}\label{ruintime1}
T=\inf\{n :U_n<0\}
\end{equation}
with the use of the convention $\inf \emptyset=\infty$. Formula \eqref{ruintime1} can be rewritten in the alternative form 

\begin{equation}\label{ruintime2}
T=\inf\{n :Y_n>U_0\}
\end{equation}
as is proven in \cite{nyrhinen2011stochastic}. These two representations of ruin time act as a link between the processes $(U_n)$ and $(Y_n)$. Clearly, $\{T<\infty\}=\{\bar{Y}>U_0\}$ and $\{T\leq n\}=\{\bar{Y}_n>U_0\}$.

Next, we recall some general notations. Let $x^+$ denote the positive part $\max(0,x)$ of a real number $x$.  If $X$ is a random variable, the quantity $\I(X)\in [0,\infty]$ is defined by
\begin{equation}\label{momentindexdef}
\I(X)=\sup\{s\geq 0: E((X^+)^s)<\infty\}.
\end{equation}
Quantity $\I(X)$ is called \emph{the moment index} of the random variable $X$. Equation \eqref{momentindexdef} is reasonable because the function $E((X^+)^s)$ is a convex function on $\mathbb{R}$. For a comprehensive treatment of convex functions, see \cite{rockafellar1997convex}.

We say that a random variable $X$ is (right) \emph{heavy-tailed} if no exponential moments exist. That is, if $\sup\{s\geq 0: E(e^{sX})<\infty\}=0$. We note that the moment index of a random variable  may obtain a finite value only when $X$ is heavy tailed.  In fact, moment index is a risk measure of heavy tailed distributions. 

As a general rule one may argue that a smaller moment index means larger risk in the sense of a heavier tail. A heavy tail on the other hand implies the increased possibility of very large realisations. From this fact we see that the moment index is important in questions related to ruin. For an introduction to heavy-tailed and especially subexponential distributions, the reader is referred to \cite{foss2011introduction}. Furthermore, to see how such distributions affect the behaviour of random walks, see \cite{borovkov2008asymptotic}.

It is known from \cite{daley2006moment} Lemma BDK or \cite{baltrunas2004tail} that there is a connection between the moment index and the tail of a random variable. Namely,
\begin{equation}\label{tailconnection}
\limsup_{x \to \infty} \frac{\log \overline{F}_{X}(x)}{\log x}=-\I(X),
\end{equation}
where the function $\overline{F}_{X}$ is the tail function $1-F_X$ of random variable $X$ and $F_X$ is the distribution function of $X$. Here the convention $\log 0=-\infty$ is used. Using the terminology of \cite{bingham1989regular}, $-\I(X)$ corresponds to the upper order of the function $\overline{F}_{X}$.

Another quantity, which is closely related to the moment index, can be defined for a non-negative random variable $Y$ by the formula
\begin{equation}\label{lundbergdef}
\I^1(Y)=\sup\{s\geq 0: E(Y^s)\leq 1\}.
\end{equation}
We note that inequality $\I^1(Y)\leq \I(Y)$ holds between \eqref{momentindexdef} and \eqref{lundbergdef}. Quantities similar to \eqref{lundbergdef} are important in classical ruin theory as well as in the theory of large deviations. For a thorough treatment of the theory of large deviations the reader is advised to see \cite{dembo2009large}. In section \ref{sublund} below the quantity $\I^1(Y)$ is characterised via moment indices. 

\subsection{The aim of the study}

Our main goal is to discover the value of
\begin{equation}\label{promo1}
\I(\bar{Y})=-\limsup_{U_0 \to \infty} \frac{ \log P(T<\infty)}{\log U_0}
\end{equation}
using as few assumptions as possible. The range in which
\begin{equation}\label{promo2}
\I(\bar{Y}_n)=-\limsup_{U_0 \to \infty} \frac{ \log P(T\leq n)}{\log U_0}
\end{equation}
lies will also be discovered.

The quantities \eqref{promo1} and \eqref{promo2} describe the rough asymptotic behaviour of the probabilities $P(T<\infty)$ and $P(T\leq n)$. Unfortunately, in general, there are no guarantees that the limes superiors of \eqref{promo1} and \eqref{promo2} coincide with the corresponding limes inferiors. However, from the viewpoint of insurance, the upper estimate obtained from these limes superiors is actually the more important one. 

To see this, suppose the limes superior of \eqref{promo1} satisfies $\limsup_{U_0 \to \infty} \log P(T<\infty)/\log U_0=-\alpha$, where $\alpha \in (0,\infty)$. Now, for any $\epsilon>0$ the estimate $P(T<\infty)\leq U_0^{-\alpha+\epsilon}$ holds for all $U_0$ large enough. Since the upper bounds for ruin probabilities are the most important, limes superior offers sufficient knowledge for this application.

We will show which quantities determine the value of $\alpha$. It is natural to think that the risk determining the value of \eqref{promo1} is the dominating risk. In this sense our study reveals which risk dominates. It turns out that the value of $\alpha$ related to the unbounded interval is determined by the smaller of the quantities $\I^1(A)$ and $\I(B)$. Our approach utilises the monotonicity properties of the model. This enables us to cut ourselves loose from the usual, yet unneeded, technical assumptions such as $\I^1(A)<\I(A)$ or $\I^1(A)<\I(B)$.

The structure of the article is built to support the main results presented in Sections \ref{estimation}-\ref{mainproof}. Section \ref{prelims} includes the necessary preliminaries and develops the theory of moment indices. Section \ref{estimation} is dedicated to the different estimation methods that follow from the monotonicity of the model. Section \ref{chaz} clarifies the assumptions needed in the proof of main theorem presented in Sections \ref{finitecase} and  \ref{mainproof}. Finally, Section \ref{conclusions} contains comments and conclusions.

\subsection{Related studies}\label{relatedstudies}

The closest study to our purpose is \cite{Harri2001265}, where a similar problem is partially considered in Theorem 2. The present paper proves that the result concerning the limit supremum of ultimate ruin holds in a more general setting.

There exists a multitude of papers that study the model \eqref{modeldesc} when further assumptions on the distributions of $A$ and $B$ are made. Studies \cite{Tang2003299} and \cite{randomsum} use the model \eqref{modeldesc}. However, these papers concentrate on sharp asymptotic results in a model where more detailed assumptions are made on the distributions of random variables.

The connection of limiting distributions and stochastic fixed point equations is used in most of the cited articles. However, to our knowledge these equations have not been combined with moment indices. A recent article \cite{nyrhinen2011stochastic} considers stochastic equations satisfied by the random variables $\bar{Y}$ and $\bar{Y}_n$. In this paper we will use these connections to find a representation for $\I(\bar{Y})$ and $\I(\bar{Y}_n)$.

Our result is about the maximal moment of random variable $\bar{Y}$. A similar viewpoint has been studied in  \cite{alsmeyer2009distributional} where the absolute moments of process \eqref{infinitesumvariable} are explored. The article \cite{alsmeyer2009distributional} raises a question concerning the moment properties of $\bar{Y}$. In the present paper this question is answered in the case of heavy tailed risks. 

It is possible to use the connection 
\begin{equation}\label{inftysup}
P(\bar{Y}>u)\geq P(Y_\infty>u)\geq P(Y_\infty>0)P(\bar{Y}>u)
\end{equation} 
mentioned in remark 2.1 of \cite{nyrhinen2011stochastic} to see that $\I(\Bar{Y})=\I(Y_\infty)$, when the series $Y_\infty$ converges and $P(Y_\infty>0)>0$. The exact conditions that ensure the almost sure existence of the limit variable \eqref{infinitesumvariable} can be found from \cite{goldie1991implicit} and \cite{goldie2000stability}. For convergence in IID-case it suffices that $E(\log^+ |B|)<\infty$ and $E( \log A)<0$ by  \cite{embrechts1997modelling}, Section 8.4. 

Outside the field of insurance, similar problems are encountered in statistical ARCH and GARCH models, as briefly mentioned in Section \ref{notmot}, as well as in queuing theory. This aspect is studied in \cite{collamore2009random}, where dependence assumptions are further relaxed by introducing a Markovian dependence structure and using techniques developed in \cite{nummelin1984general}.

\section{Tools for analysis}\label{prelims}

\subsection{Stochastic equations}\label{stochdiff}

The process \eqref{modeldesc} has many comfortable properties in terms of distributional equations. More precisely, random variables \eqref{ultimatevariable} and \eqref{finitevariable} satisfy useful random difference and fixed point equations.

Consider a random variable defined at discrete time points $n$ recursively by the formula
\begin{equation}
V_0=0, \ V_n\stackrel{d}{=}B_n+A_n V_{n-1}^+, \ n\geq 1,
\end{equation}
where the random vector $(A_n,B_n)$ is independent of $V_{n-1}$. Since
\begin{eqnarray*}
\bar{Y}_n&=&\max(B_1,B_1+A_1B_2,\ldots,B_1+A_1B_2+\ldots+A_1\ldots A_{n-1} B_n) \\
&=&B_1+A_1\max(0,B_2,\ldots,B_2+A_2B_3+\ldots+A_1\ldots A_{n-1} B_n),
\end{eqnarray*}
it is easy to see that $\bar{Y}_n \stackrel{d}{=}V_n$ for all $n \in \mathbb{N}$. This connection allows us to calculate some of the moment indices recursively. 

In following formulas \eqref{yhtalo1} and \eqref{yhtalo2} the vector $(A,B)$ is independent of $\bar{Y}$ or $Y_\infty$ on the right hand side. It is known from \cite{nyrhinen2011stochastic} that
\begin{equation}\label{yhtalo1}
\bar{Y}\stackrel{d}{=}B+A\bar{Y}^+.
\end{equation}
In addition, the random variable $Y_\infty$ representing the perpetuity is well defined under mild conditions, as mentioned in Section \ref{relatedstudies}. In this case $Y_\infty$ satisfies a simple equation
\begin{equation}\label{yhtalo2}
Y_\infty \stackrel{d}{=}B+AY_\infty.
\end{equation}
Equation \eqref{yhtalo2} is often useful because under general assumptions the equality 
$$\I(Y_\infty)=\I(\bar{Y})$$ holds via Equation \eqref{inftysup}.

\subsection{Properties of moment indices}

Apart from \cite{daley2001moment} and \cite{daley2006moment}, sources for the use and properties of moment indices have been scarce. Hence, we take the opportunity to present some general results.

Much of the following deduction is based on the fact that random variables with the same distributions share a common moment index. This result can be seen directly.

\begin{lemma}\label{yhteyslause} Let $X$ and $Y$  be random variables such that $X \stackrel{d}{=} Y$. Then
\begin{equation*}
\I(X)=\I(Y). 
\end{equation*}
\qed
\end{lemma}

It is known from \cite{Tang2003299} that $\I(X+Y)=\I(XY)=\min(\I(X),\I(Y))$, if the random variables $X$ and $Y$ are independent and one of them is non-negative (but not identically $0$ in the case of product). We will develop similar results under more lenient assumptions. Especially assumptions on independence are alleviated. This is important when one considers vectors, whose components are not independent. 

We begin with a list of the most basic properties of moment indices. 

\begin{lemma}\label{lin}
Let $X$ and $Y$ be random variables such that $Y \leq X$ almost surely. Assume $a>0$ and $b \in \R$. Then
\begin{enumerate}
\item \label{kohta1}
\begin{equation*}
\I(aX+b)=\I(X)
\end{equation*}
\item \label{kohta2}
\begin{equation*}
\I((X^+)^a)=\frac{\I(X)}{a}
\end{equation*}
\item \label{kohta3}
\begin{equation*}
\I(X)\leq\I(Y)
\end{equation*}
\item \label{kohta4}
\begin{equation*}
\I(X\mathbf{1}(X>b))=\I(X).
\end{equation*}
\end{enumerate}
\end{lemma}

\begin{proof}
Part \ref{kohta1} is clear, while parts \ref{kohta2} and \ref{kohta3} follow from the definition of moment index \eqref{momentindexdef}. Part \ref{kohta4} is obvious by \eqref{tailconnection}.
\end{proof}

The following lemma presents a general property of moment index.
\begin{lemma}\label{maxlemma}
Let $X$ and $Y$ be random variables. Then
\begin{equation*}
\I(\max(X,Y))=\min(\I(X),\I(Y)).
\end{equation*}
\end{lemma}

\begin{proof} Clear by part \ref{kohta3} of Lemma \ref{lin}. 
\end{proof}

The next lemma states two important rules of moment indices. They are partially known from \cite{Tang2003299}.

\begin{lemma}\label{tokalemma}
Let $X$ and $Y$ be random variables. Then
\begin{equation}\label{zemma1}
\I(X+Y)\geq \min(\I(X),\I(Y)).
\end{equation}
Furthermore, if $X$ and $Y$ are independent or non-negative,
\begin{equation*}
\I(X+Y)= \min(\I(X),\I(Y)).
\end{equation*}
\end{lemma}

\begin{proof}
If $\min(\I(X),\I(Y))=0$ there is nothing to prove in \eqref{zemma1}. Assume that there exists $0< s<\min(\I(X),\I(Y))$. Since
\begin{eqnarray*}
E((X+Y)^s \mathbf{1}(X+Y>0)) &\leq & E((2\max(X,Y))^s\mathbf{1}(X+Y>0)) \\
& \leq & 2^s E((\max(X,Y))^s\mathbf{1}(\max(X,Y)>0))<\infty, 
\end{eqnarray*}
it must hold by Lemma \ref{maxlemma} that $\I(X+Y)\geq \I(\max(X,Y))$.

Assume then that the random variables $X$ and $Y$ are independent. As the other direction follows directly from Equation \eqref{zemma1} it suffices to prove the inequality $\I(X+Y)\leq \min(\I(X),\I(Y))$. To do this, we begin by choosing a real number $b$ such that $P(X>b)>0$. Now, for any $0<s<\I(X+Y)$, we get
\begin{eqnarray*}
E(((X+Y)^+)^s) & \geq & E(((X+Y)^+)^s\mathbf{1}(X>b)) \\
& \geq &  E(((b+Y)^+)^s\mathbf{1}(X>b)) \\
& = & E(((b+Y)^+)^s)P(X>b). 
\end{eqnarray*}
By Lemma \ref{lin} part \ref{kohta1} $\I(Y+b)=\I(Y)$ holds and we see that $\I(X+Y)\leq \I(Y).$ Using the symmetry of the random variables $X$ and $Y$ we also obtain $\I(X+Y)\leq \I(X)$. Therefore $\I(X+Y)\leq \min(\I(X),\I(Y))$.

Assume finally that random variables $X$ and $Y$ are non-negative. Now the events $\{X\geq 0\}$, $\{Y\geq 0\}$ and $\{X+Y\geq 0\}$ are almost sure events. Hence
\begin{eqnarray*}
E((X+Y)^s\mathbf{1}(X+Y>0))&=&E((X+Y)^s) \\
&\geq & \max(E(X^s),E(Y^s)) \\
&=&\max(E((X\mathbf{1}(X>0))^s),E((Y\mathbf{1}(Y>0))^s)),
\end{eqnarray*}
which proves the claim. 
\end{proof}

It is useful to notice the following property of moment indices and the related corollary.

\begin{lemma}\label{separation} Let $X$ be a random variable on probability space $(\Omega,\mathcal{F},P)$ and $S \in \mathcal{F}$. Then
\begin{equation*}
\I(X)=\min(\I(X \mathbf{1}_S),\I(X \mathbf{1}_{S^c})).
\end{equation*}
\end{lemma}
\begin{proof} Using the decomposition
\begin{equation*}
X^s\mathbf{1}(X>0)=X^s\mathbf{1}(X>0)\mathbf{1}_S+X^s\mathbf{1}(X>0)\mathbf{1}_{S^c}
\end{equation*}
and positivity we may estimate
\begin{equation}\label{eka}
\max[E(X^s\mathbf{1}(X>0)\mathbf{1}_S),E(X^s\mathbf{1}(X>0)\mathbf{1}_{S^c})]\leq E(X^s\mathbf{1}(X>0)).
\end{equation}
By Equation \eqref{eka} we see that $\I(X)\leq \min(\I(X \mathbf{1}_S),\I(X \mathbf{1}_{S^c}))$. The other direction is immediately valid because of \eqref{zemma1}.  
\end{proof}

As a direct consequence of Lemma \ref{separation} we obtain the following corollary.
 
\begin{corollary}\label{jakokorollaari}
Let $\Theta$ be a finite cover of the state space $\Omega$. Then by Lemma \ref{separation}
\begin{equation*}
\I(X)=\min_{\theta \in \Theta} \I(X\mathbf{1}_\theta).
\end{equation*} 
\end{corollary}

Corollary \ref{jakokorollaari} motivates to define a new conditional quantity. If $X$ is a random variable and $H$ an event of the state space $\Omega$ with $P(H)>0$, we define \emph{the conditional moment index} $\I(X|H)$ of random variable $X$ by setting
\begin{equation*}
\I(X|H)=\sup\{s\geq 0 : E((X^+)^s|H)<\infty \}=\sup\{s\geq 0 : E((X^+)^s\mathbf{1}_H)<\infty \}.
\end{equation*}

We end the section with an example that shows how the stochastic fixed point equations can be used with moment indices.

\begin{example}\label{ex3} Consider the model \eqref{modeldesc} satisfying assumptions \ref{as3}-\ref{as2} of Section \ref{alkuassumptions}. Assume further that the sequence $(B_i)$ consists of non-negative random variables. Then we may deduce an upper bound for the moment index of random variable $\bar{Y}$. Namely,
\begin{equation}\label{ex2}
\I(\bar{Y})\leq \min(\I(A),\I(B)).
\end{equation}

Equation \eqref{ex2} is a direct consequence of the stochastic equation the random variable $\bar{Y}$ satisfies. We recall from \eqref{yhtalo1} that $\bar{Y}\stackrel{d}{=}B+A\bar{Y}^+$. By Lemmas \ref{yhteyslause} and \ref{tokalemma} we get 
\begin{equation}\label{ex22}
\I(\bar{Y})=\min(\I(A),\I(B),\I(\bar{Y})).
\end{equation}
Obviously \eqref{ex22} is equivalent to \eqref{ex2}.

Two remarks are in order. Firstly, a similar bound could be obtained along the same line of thought if we assumed that the sequence $(A_i)$ was independent of $(B_i)$. In this case insurance risks could be real valued random variables. Secondly, the bound obtained in \eqref{ex2} is not the best possible. In the case $\I^1(A)<\min(\I(A),\I(B))$ the best bound is $\I^1(A)$. This will become clear after the proof of the main theorem in Section \ref{mainproof}.
\end{example}

\section{Estimation methods}\label{estimation}

We will study how the process \eqref{modeldesc} changes when the original risks are replaced by almost surely smaller risks. The new risks, as well as all quantities related to them, are marked with the asterisk symbol $*$.

\subsection{Monotonicity of the financial risk}\label{monotonyoffinance}

Let us assume that $A$ is replaced by random variable $A^*$, for which $A\geq A^*>0$. Assume further that every other aspect of the model remains unchanged. Especially the insurance risk $B$ has the same marginal distribution as before.

We will examine the relationship between the ruin times $T$ and $T^*$. Recall that the capital at the end of the year $n$ is given by the random variable $U_n$ that satisfies the recursion \eqref{modelrec}. We will prove the inequality
\begin{equation}\label{Uarvio}
U_n\leq U_n^*
\end{equation}
for every $n$ before the time of ruin, when $A_n\geq A_n^*$, that is, $r_n\leq r_n^*$.

Formula \eqref{Uarvio} can be justified using the following argument. We note first that before the time of ruin $T$ each of the random variables $U_{n-1}-B_n$ is positive. If $T=1$, it follows that $U_1^*\leq U_1<0$ and therefore $T^*=T$. Now, if $T \geq 2$, we may use inductive reasoning for each point of the sample space.
\begin{enumerate}
\item If $T=k$, where $k\geq 2$, at time $n=1$ the inequality
\begin{equation*}
U_1=(1+r_1)(U_0-B_1)\leq (1+r_1^*)(U_0-B_1)=U_1^*,
\end{equation*}
holds and \eqref{Uarvio} is valid. 

\item Assume that \eqref{Uarvio} holds for all $1 \leq n \leq T-2$. Then 
\begin{eqnarray*}
U_{n+1}=(1+r_{n+1})(U_n-B_{n+1})& \leq &(1+r_{n+1}^*)(U_n-B_{n+1}) \\
& \leq & (1+r_{n+1}^*)(U_n^*-B_{n+1}).
\end{eqnarray*}
\end{enumerate}
The order of $U$ and $U^*$ cannot be deduced exactly at the time of ruin. However, it is possible to infer that the process $(U_n^*)$ cannot obtain its first negative value before the process $(U_n)$ has done so. Hence $T^*\geq T$, when $A\geq A^*$.

\subsubsection{Case $\I^1(A)=0$ by estimation of financial risk $A$}\label{arviokappale}

The result $T\leq T^*$ from Section \ref{monotonyoffinance} allows us to deduce 
\begin{equation*}
P(T_{U_0}<\infty)\geq P(T^*_{U_0}<\infty).
\end{equation*}
This in turn enables the estimate
\begin{equation}\label{hassujuttu1}
\limsup_{U_0 \to \infty} \frac{\log P(T_{U_0}<\infty)}{\log U_0}\geq \limsup_{U_0 \to \infty} \frac{\log P(T^*_{U_0}<\infty)}{\log U_0}.
\end{equation}
The following lemma shows that the needed random variable $A^*$ exists.

\begin{lemma}\label{aarviolemma} Let $A>0$ be a random variable such that $\I^1(A)<\infty$. Fix $\epsilon>0$. Then there exists a bounded random variable $A_\epsilon$ that fulfils the conditions
\begin{enumerate}
\item $A\geq A_\epsilon>0$ almost surely and \label{ekaehto}
\item $\I^1(A_\epsilon)=\I^1(A)+\epsilon$. \label{tokaehto}
\end{enumerate}
\end{lemma}

\begin{proof} Define
$$m_\epsilon=\sup\left\{m: E\left( A^{\mathbb{I}^1(A)+\epsilon} \mathbf{1}(A\leq m) \right)\leq 1 \right\}.$$
In case $P(A=m_\epsilon)=0$, it must hold that $E\left( A^{\mathbb{I}^1(A)+\epsilon} \mathbf{1}(A\leq m_\epsilon) \right)=1$. For this case we define 
$$A'=A \mathbf{1}(A \leq m_\epsilon). $$
In case $P(A=m_\epsilon)>0$ and $E\left( A^{\mathbb{I}^1(A)+\epsilon} \mathbf{1}(A\leq m_\epsilon) \right)>1$, then for a suitably chosen constant $c_1 \in [0,m_\epsilon)$ it holds that $E\left( A^{\mathbb{I}^1(A)+\epsilon} \mathbf{1}(A< m_\epsilon) \right)+ c_1^{\mathbb{I}^1(A)+\epsilon} P(A=m_\epsilon)=1$. Then we define 
$$A'=A \mathbf{1}(A<m_\epsilon)+c_1 \mathbf{1}(A=m_\epsilon). $$
Let $c_2>0$ be so small that the random variable
$$\hat{A}=A'+c_2\mathbf{1}(A>m_\epsilon) $$
satisfies $\I^1(\hat{A})\in (\mathbb{I}^1(A)+\epsilon/2,\mathbb{I}^1(A)+\epsilon)$. Setting $A_\epsilon=c_3\hat{A}$ for a suitable constant $c_3\in (0,1)$ ends the construction. 
\end{proof}

\subsection{Monotonicity of insurance risk}

Insurance risks have a similar monotonicity property as financial risks. This is clear, since from $B^*\leq B$ it follows that $Y_k^*\leq Y_k$ for all $k$. Hence $\bar{Y}^*\leq \bar{Y}$ and $P(T_{U_0}<\infty)\geq P(T_{U_0}^*<\infty)$.

The next result demonstrates how the moment index of a random variable can be increased by estimation from below.

\begin{lemma}\label{omamoma}
Let $B$ be a real valued random variable. Assume $\I(B)=\alpha<\infty$ and $\beta>\alpha$.

Then there exists a random variable $B^*$ such that almost surely $B^*\leq B$ and $\I(B^*)=\beta$.
\end{lemma}

\begin{proof}
Let $W$ be a Pareto distributed random variable independent of $B$ with parameter $\beta-\alpha$. By a Pareto distributed random variable with parameter $\gamma>0$ we mean a random variable whose tail function is
\begin{displaymath}
   P(W>x) = \left\{
     \begin{array}{rl}
       \frac{1}{x^{\gamma}} & : x\geq 1 \\
       1& : x<1.
     \end{array}
   \right.
\end{displaymath}
Set $B^*=\min(B,W)$. Now
\begin{eqnarray*}
\limsup_{x \to \infty} \frac{\log \overline{F}_{B^*}(x)}{\log x} &=& \limsup_{x \to \infty} \left( \frac{\log P(B>x)}{\log x}+\frac{\log P(W>x)}{\log x} \right) \\
&=& \limsup_{x \to \infty} \frac{\log P(B>x)}{\log x}+\lim_{x \to \infty}\frac{\log P(W>x)}{\log x} \\
&=& -\alpha-(\beta-\alpha)=-\beta,
\end{eqnarray*}
which ends the construction. 
\end{proof}

\section{Essential suprema of $\bar{Y}_k$ and $\bar{Y}$}\label{chaz}

Let us recall that for a general random variable $X$ the \emph{essential supremum} is defined as
\begin{eqnarray*}
\esssup X=\esssup_{\omega \in \Omega} X(\omega)&=&\inf\{a \in \mathbb{R}:P(X>a)=0\} \\
&=& \sup \{ a \in \mathbb{R}: P(X>a)>0 \}.
\end{eqnarray*} 

We set $\bar{y}=\esssup \bar{Y}$ and define in an analogous way a number that describes the essential supremum of the random variable $\bar{Y}_k$. Put $\bar{y}_0=0$ and $\bar{y}_k=\esssup \bar{Y}_k$, when $k \in \mathbb{N}$. The following analysis will reveal how the sequence $(\bar{y}_k)$ behaves.

The dependence structure of the process $(Y_k)$ may lead to unexpected behaviour. The following example demonstrates how the whole process may be limited almost surely for the first $N+1$ years and then become essentially unbounded.

\begin{example}\label{maaresim}
Suppose $N \in \mathbb{N}$ is fixed. Let $W_\gamma$, for $\gamma>0$, be a Pareto distributed random variable with parameter $\gamma$. Let $\alpha>0$ and let $K$ be independent of $W_\alpha$ with $P(K=0)=P(K=1)=\frac{1}{2}$. We define $(A,B)$ by
\begin{equation*}
B=\mathbf{1}(K=0)-NW_\alpha \mathbf{1}(K=1)
\end{equation*} 
and
\begin{equation*}
A=\mathbf{1}(K=0)+W_\alpha \mathbf{1}(K=1).
\end{equation*}

Clearly $\bar{y}_k=k$ for all $k=1,2,\ldots ,N+1$. However, $\bar{y}_{N+2}=\infty$.
\end{example}

The moment index of a bounded random variable is infinite. This is why a systematic method of determining the value of $\bar{y}_k$ with respect to a given vector $(A,B)$ is needed. The next theorem summarises how the sequence $(\bar{y}_k)$ behaves and offers a new characterisation for the condition $\bar{y}=\infty$.
\begin{theorem}\label{viit} Assume that the generic random vector $(A,B)$ satisfies
\begin{enumerate}[(i)]
\item $P(A>0)=1$ \label{oz1}
\item $P(A>1)>0$ \label{oz2}
\item $P(A<1)>0$ \label{oz3}
\item $P(B>0)>0$.\label{oz4}
\end{enumerate}
Then, for any $N \in \mathbb{N}$,
\begin{equation}\label{rajoitettuyy}
\bar{y}_N=\esssup (B+\bar{y}_{N-1}A)
\end{equation}
and the following conditions are equivalent:
\begin{enumerate}
\item $\bar{y}<\infty$ \label{gohta1}
\item $\lim_{N \to \infty} \bar{y}_N<\infty$ \label{gohta2}
\item There exists $c>0$ such that $P(B+cA\leq c)=1$. \label{gohta3}
\item For all $k \in \mathbb{N}$ equality $P(Y_k>0,A_1\ldots A_k>1)=0$ holds. \label{gohta4}
\end{enumerate}
\end{theorem}
\begin{proof} We begin by showing equality \eqref{rajoitettuyy}. Denote $\esssup (B+\bar{y}_{N-1}A)=\kappa$ and define function $g \colon [0,\infty) \to \mathbb{R}\cup\{\infty\}$ by formula
\begin{equation*}
g(c)=\esssup (B+A c).
\end{equation*}

We recall from Section \ref{stochdiff} that $\bar{Y}_N \stackrel{d}{=} V_N$, where
\begin{equation*}
V_0=0, \ V_N\stackrel{d}{=}B_N+A_N \max(0,V_{N-1}).
\end{equation*}
Because $\esssup V_N= \esssup \bar{Y}_N$ and the risks of different years are independent, we get
\begin{eqnarray*}
\bar{y}_N&=& \esssup \bar{Y}_N \\
&=& \esssup V_N \\
&=& \esssup (B_N+A_N V_{N-1}) \\
& \leq & \esssup g(V_{N-1}) \\
& \leq & g(\bar{y}_{N-1}).
\end{eqnarray*}
This yields $\bar{y}_N \leq \kappa$. 

For the remaining direction, assume first $\kappa<\infty$. Fix $\epsilon>0$. Now there exists a set $H_1 \in \Omega$ for which $P(H_1)>0$ and
\begin{equation*}
B(\omega)+\bar{y}_{N-1} A(\omega)>\kappa-\epsilon, \ \forall \omega \in H_1.
\end{equation*}
We may again choose a number $\eta=\eta_\epsilon>0$ in a way that there is a set $H_2\subset H_1$, where $P(H_2)>0$, and
\begin{equation*}
B(\omega)+(\bar{y}_{N-1}-\eta) A(\omega)>\kappa-\epsilon, \ \forall \omega \in H_2.
\end{equation*}
Using the definition of supremum and the independence structure of the process $(Y_k)$ we can find a third set $H_3\subset H_2$, where $P(H_3)>0$, and
\begin{equation*}
B(\omega)+\bar{Y}_{N-1}(\omega) A(\omega)>\kappa-\epsilon, \ \forall \omega \in H_3.
\end{equation*}
This implies $\kappa-\epsilon \leq \bar{y}_{N}$ by using the connections between variables $\bar{Y}_{N-1}$ and $V_{N-1}$. Equation \eqref{rajoitettuyy} then follows by letting $\epsilon \to 0$. Suppose then that $\kappa=\infty$. The above deduction can now be done again by replacing $\kappa-\epsilon$ with an arbitrarily large number $M$. This gives the remaining result.

We proceed to the proof of equivalences \ref{gohta1}-\ref{gohta4}. The equivalence of \ref{gohta1} and \ref{gohta2} is clear, since $(\bar{y}_N)$ is a non-decreasing sequence of numbers whose limit is $\bar{y}$.

Assume \ref{gohta2}. Denote $\lim_{N \to \infty} \bar{y}_N=c$. By assumption \ref{oz4} inequality $c>0$ holds. Using the connection
\begin{equation*}
\bar{y}_{N+1}=\esssup (B+\bar{y}_{N}A)
\end{equation*}
we may deduce that
\begin{equation*}
c=\esssup (B+cA).
\end{equation*}
Hence almost surely
\begin{equation*}
B+cA\leq c
\end{equation*}
and \ref{gohta3} is valid.

Assume \ref{gohta3}. Set $a_-=\essinf A$, where the essential infimum is defined analogously to the essential supremum. By assumptions \ref{oz1} and \ref{oz3} we see that $a_- \in [0,1)$. Now, almost surely
\begin{equation*}
B\leq c(1-A) \leq c(1-a_-).
\end{equation*}
Especially $\bar{y}_1=\esssup B\leq c(1-a_-)$. Therefore 
\begin{eqnarray*}
\bar{y}_2&=&\esssup (B+\bar{y}_1 A) \\
&\leq &\esssup (B+c(1-a_-) A) \\
&=&\esssup (B+cA-ca_-A) \\
&\leq &c+\esssup (-ca_-A) \\
&=&c-c(a_-)^2=c(1-(a_-)^2).
\end{eqnarray*}
In general 
\begin{equation*}
\bar{y}_N\leq c(1-(a_-)^N)\to c<\infty,
\end{equation*}
holds when $N \to \infty$. This proves \ref{gohta2}.

Equivalence of \ref{gohta1} and \ref{gohta4} is precisely the content of theorem first proved in \cite{nyrhinen2011stochastic}, Theorem 2. 
\end{proof}

Conditions \ref{gohta3} and \ref{gohta4} of Theorem \ref{viit} can be regarded as microscopic and macroscopic ways to see when the essential supremum of the process $(Y_k)$ is unlimited. The first of these tells us what the random vectors $(A,B)$ must satisfy. The latter condition looks at the process $(Y_k)$ on a large scale. Condition \ref{gohta3} is an operational tool that can be used in proofs, whereas condition \ref{gohta4} is difficult to verify in practice.

\section{The case of bounded interval}\label{finitecase}

We will begin with a simple observation based on Corollary \ref{jakokorollaari}.

\begin{lemma}\label{jakolemma} Suppose $N \in \mathbb{N}$. Then
\begin{equation}\label{jakolemma1}
\I(\bar{Y}_N)=\min_{1\leq k\leq N} \left( \min_{\theta \in \Theta_k} \I(Y_k|\theta)\right),
\end{equation}
where
\begin{equation}\label{theeta}
\Theta_k=\{S \subset \Omega:S=\cap_{i=1}^k K_i, \ \mbox{where} \ K_i=\{B_i>0\} \ \mbox{or} \ K_i=\{B_i\leq 0\}\}.
\end{equation}
\end{lemma}
\begin{proof}
Using Lemma \ref{maxlemma} it is clear that $\I(\bar{Y}_N)=\min_{1\leq k \leq N} \I(Y_k)$. For a fixed $k$ the family $\Theta_k$ of equation \eqref{theeta} forms a partition of the state space $\Omega$. Hence, by using Corollary \ref{jakokorollaari}, Equation \eqref{jakolemma1} holds. 
\end{proof}

In general, when dependence between the random variables $A$ and $B$ is possible, it is not clear in which set the minimum \eqref{jakolemma1} is obtained. An especially interesting interaction between $A$ and $B$ can be observed if the insurance risk partially cancels the financial risk. 

To understand this phenomenon, we define the function $h \colon [0,\infty] \to [0,\infty]$ by the formula
\begin{equation}\label{ihqfunktio}
h(c)=\I(B+cA),
\end{equation} 
when $0\leq c<\infty$. Because the estimate $B+c_1A\leq B+c_2A$ holds for each $0\leq c_1<c_2$, we get $\I(B+c_1A)\geq \I(B+c_2A)$ and deduce that $h$ is a decreasing function. In addition, $h$ is limited from below by $0$. Hence, we can define
\begin{equation*}
h(\infty)=\lim_{c \to \infty} h(c).
\end{equation*}

It turns out the function $h$ of formula \eqref{ihqfunktio} offers enough information from the random vector $(A,B)$ in order to express the quantities $\I(\bar{Y}_k)$. The following two examples demonstrate that the function $h$ is not necessarily right or left continuous.\footnote{The fact that $h$ is not necessarily left continuous was kindly pointed out by Harri Nyrhinen in a personal communication. Essentially, this idea is formulated in example \ref{eivasjva}.}

\begin{example}\label{hassuesim5} Suppose $W_\gamma$ is a Pareto distributed as in example \ref{maaresim}. Let $W_\alpha$ and $W_\beta$ be independent with $0<\beta<\alpha$. Assume further that $K$ is independent of $W_\alpha$ and $W_\beta$ with $P(K=0)=P(K=1)=\frac{1}{2}$. We define $(A,B)$ by
\begin{equation*}
B=-W_\alpha \mathbf{1}(K=0)-2W_\beta \mathbf{1}(K=1)+10
\end{equation*}
and
\begin{equation*}
A=W_\alpha \mathbf{1}(K=0)+W_\beta \mathbf{1}(K=1).
\end{equation*}
Now 
\begin{displaymath}
   \I(B+cA)= \left\{
     \begin{array}{ll}
       \infty & : c \in [0,1]\\
       \alpha & : c \in (1,2]\\
       \beta & : c>2.
     \end{array}
   \right.
\end{displaymath}
\end{example}

Assuming $B$ is a sum of suitable dependent random variables enables us to see the following.

\begin{example}\label{eivasjva} Suppose $W_\gamma$ is defined as in example \ref{maaresim} and $1<\beta<\alpha$. Set $A=W_{\beta/2}+W_\alpha$ and $B=-W_{\beta/2}+W_{\beta/2}^{1/2}+1$, where $W_{\beta/2}$ and $W_{\alpha}$ are independent. Noting that $(c-1)W_{\beta/2} +W_{\beta/2}^{1/2}$ is bounded for $c\in (0,1)$ and using results of Lemmas \ref{lin} and \ref{tokalemma} we obtain
\begin{displaymath}
   \I(B+cA)= \left\{
     \begin{array}{ll}
       \alpha & : c \in (0,1)\\
       \beta & :  c=1.
     \end{array}
   \right.
\end{displaymath}
\end{example}

We may now give a range in which $\I(\bar{Y}_k)$ always is.

\begin{theorem}\label{finitetheorem} Consider the process $(Y_k)$ of \eqref{modeldesc} consisting of IID vectors $(A,B)$, where $A>0$. Then, for any fixed $N \in \mathbb{N}$,
\begin{equation}\label{todistettavax}
\I(\bar{Y}_N)\in \Big{[}h(\bar{y}_{N-1}),\lim_{c \to \bar{y}_{N-1}-} h(c)\Big{]},
\end{equation}
where $h$ is the function defined in \eqref{ihqfunktio}.
\end{theorem}
\begin{proof}
To prove \eqref{todistettavax} we note that for any $\epsilon>0$ and arbitrary $s<\I(\bar{Y}_N)$
\begin{eqnarray*}
E((\bar{Y}_N)^s\mathbf{1}(\bar{Y}_N>0))&=& E(((B+A\bar{Y}_{N-1}^+)^+)^s) \\
&\geq & E(((B+A\bar{Y}_{N-1}^+)^+)^s\mathbf{1}(\bar{Y}_{N-1}>\bar{y}_{N-1}-\epsilon)) \\
&\geq & E(((B+A(\bar{y}_{N-1}-\epsilon))^+)^s\mathbf{1}(\bar{Y}_{N-1}>\bar{y}_{N-1}-\epsilon)) \\
&=& E(((B+A(\bar{y}_{N-1}-\epsilon))^+)^s)P(\bar{Y}_{N-1}>\bar{y}_{N-1}-\epsilon),
\end{eqnarray*}
which indicates that $E(((B+A(\bar{y}_{N-1}-\epsilon))^+)^s)$ cannot be infinite when $s<\I(\bar{Y}_N)$. Hence
\begin{equation}\label{jihaa}
\I(\bar{Y}_N)\leq \lim_{c \to \bar{y}_{N-1}-} h(c).
\end{equation}

On the other hand
\begin{eqnarray*}
E((\bar{Y}_N)^s\mathbf{1}(\bar{Y}_N>0))&=& E(((B+A\bar{Y}_{N-1}^+)^+)^s) \\
&\leq & E(((B+A\bar{y}_{N-1})^+)^s),
\end{eqnarray*}
so
\begin{equation}\label{jihaa2}
\I(\bar{Y}_N)\geq h(\bar{y}_{N-1}).
\end{equation}

Equations \eqref{jihaa} and \eqref{jihaa2} prove the claim. 
\end{proof}

\begin{corollary}\label{jatkokoro} Assume that
\begin{enumerate}
\item $A$ and $B$ are independent or
\item $B$ is non-negative.
\end{enumerate}
Then 
\begin{equation*}
\I(\bar{Y}_N)=\min(\I(A),\I(B))
\end{equation*}
for any $N\geq 2$.
\end{corollary}
\begin{proof} By Lemma \ref{tokalemma} and Lemma \ref{lin} part \ref{kohta1} the function $h$ of \eqref{ihqfunktio} is the constant function $\min(\I(A),\I(B))$ for all $c>0$. 
\end{proof}

What is new in Theorem \ref{finitetheorem} compared to Theorem 4.1 of \cite{Tang2003299} is that the risks of the same year may have an arbitrary dependence structure. This is possible due to the refined use of moment indices. Corollary \ref{jatkokoro} also extends the result of \cite{Tang2003299} beyond independence, if $B$ is non-negative. 

\section{The case of unbounded interval}\label{mainproof}

This section is dedicated to the proof of the last main theorem. The proof is divided into two parts. The aim of the first part is to obtain an upper bound for the left hand side of \eqref{ultimateruin} in the most general possible setting. The second part is divided into different cases.

\begin{theorem}\label{kolmastheorem} Assume that $(A_i)$ and $(B_i)$ satisfy the assumptions \ref{as3}-\ref{as2} of Section \ref{alkuassumptions}. Assume further that $\bar{y}=\infty$. Then
\begin{equation}\label{ultimateruin}
\limsup_{U_0 \to \infty} \frac{\log P(T_{U_0}<\infty)}{\log U_0}=-\min(\I^1(A),\I(B)),
\end{equation}
that is, $\mathbb{I}(\bar{Y})=\min(\mathbb{I}^1(A),\mathbb{I}(B))$.
\end{theorem}

If $Y_\infty$ of \eqref{infinitesumvariable} is well defined, we set $y_\infty=\esssup Y_\infty$ and get the following corollary. Equivalent conditions for $y_\infty=\infty$ can be found from \cite{nyrhinen2011stochastic}, Theorem 2.2.

\begin{corollary}\label{kokorrko} Suppose \ref{as3}-\ref{as2} of Section \ref{alkuassumptions} hold and $y_\infty=\infty$. Then 
\begin{equation}\label{seriesruin}
\limsup_{U_0 \to \infty} \frac{\log P(Y_\infty>U_0)}{\log U_0}=-\min(\I^1(A),\I(B)),
\end{equation}
that is, $\mathbb{I}(Y_\infty)=\min(\mathbb{I}^1(A),\mathbb{I}(B))$.
\begin{proof} We see that the assumptions of Theorem \ref{kolmastheorem} are valid. In addition, $P(Y_\infty>0)>0$ holds. Now \eqref{inftysup} yields the result. 
\end{proof}
\end{corollary}

If $B$ can take values of both signs with positive probability and conditions of Corollary \ref{kokorrko} hold, we immediately obtain
\begin{equation*}
\I(|Y_\infty|)=\min(\I(-Y_\infty),\I(Y_\infty))=\min(\I^1(A),\I(|B|)).
\end{equation*}
This is consistent with Theorem 1.4 of \cite{alsmeyer2009distributional}.

\subsection{Upper bound}\label{upperbounds}

\begin{proof}
 
Only the case $0<s<\min(\I^1(A),\I(B))$ has mathematical content. We begin by noting that
\begin{equation}\label{sharperestimate}
\bar{Y}_n\mathbf{1}(\bar{Y}_n>0) \leq \sum_{i=1}^n A_1\ldots A_{i-1}B_i^+.
\end{equation}
We will consider the following two cases. 
\begin{enumerate}
\item If $1<s$ we may use \eqref{sharperestimate} and the Minkowski's inequality to obtain
\begin{eqnarray*}
\left( E\left( \left(\bar{Y}_n\right) ^s \mathbf{1}\left(\bar{Y}_n>0\right) \right) \right)^{1/s} &\leq& \sum_{i=1}^n \left( E\left( \left(A_1\ldots A_{i-1} B_i^+\right)^s\right) \right)^{1/s} \\
&=&E\left({(B_1^+)}^s\right)\sum_{i=1}^n \left( E\left(A^s\right)\right) ^{(i-1)/s} \\
&\leq & E\left({\left(B_1^+\right)}^s\right)\sum_{i=1}^\infty \left( E\left(A^s\right)\right)^{(i-1)/s}<\infty.
\end{eqnarray*}
Now, by monotone convergence  
\begin{equation*}
E\left(\left(\bar{Y}\right)^s\mathbf{1}\left(\bar{Y}>0\right)\right)=\lim_{n \to \infty} E\left( \left(\bar{Y}_n\right)^s \mathbf{1}\left(\bar{Y}_n>0\right)\right)<\infty.
\end{equation*}

\item If $0<s\leq 1$ we may use subadditivity of the function $x \mapsto x^s$ for $x\geq 0$ and obtain similarly that $E((\bar{Y}_n)^s \mathbf{1}(\bar{Y}>0))<C$ for a constant $C$ that does not depend on $n$. This gives $E(\bar{Y}^s\mathbf{1}(\bar{Y}>0))<\infty$ as before.
\end{enumerate}

In either case
\begin{equation*}
E(\bar{Y}^s\mathbf{1}(\bar{Y}>0))\geq E(\bar{Y}^s\mathbf{1}(\bar{Y}>U_0))\geq U_0^{s}P(\bar{Y}>U_0),
\end{equation*}
which yields
\begin{equation*}
\limsup_{U_0 \to \infty} \frac{ \log P(\bar{Y}>U_0)}{\log{U_0}}\leq -s
\end{equation*}
for all $0<s<\min(\I^1(A),\I(B))$. Therefore 
\begin{equation*}
\limsup_{U_0 \to \infty} \frac{ \log P(T_{U_0}<\infty)}{\log{U_0}}=\limsup_{U_0 \to \infty} \frac{ \log P(\bar{Y}>U_0)}{\log{U_0}}\leq -\min(\I^1(A),\I(B)).
\end{equation*}
\end{proof}

\subsection{Lower bound}\label{lowerbounds}

We begin the proof of the lower bound by restricting ourselves to the most difficult situation. The remaining cases can be deduced from this result using estimation methods developed in Section \ref{estimation}.

\subsubsection{Restricted setting}\label{restricted}

We assume that 
\begin{equation}\label{assumptionsxx}
\min(\I^1(A),\I(B))>0 \ \mbox{and} \ \I^1(A)<\I(A).
\end{equation}
We will need two additional lemmas. Lemma \ref{laskulemma} is merely an observation, but Lemma \ref{alleyksilemma} presents a fundamental property of the process $(Y_k)$.

\begin{lemma}\label{laskulemma} Let $\chi$ be a real valued random variable and $\xi$ a strictly positive random variable. Then, for any $c>0$,
\begin{eqnarray}
\{\chi+\xi c>c\}&=&\{\chi>0,\xi\geq 1\} \label{joukkoesitys} \\
&\cup& \{\frac{\chi}{1-\xi}<c,\chi\leq 0,\xi> 1\} \nonumber \\
&\cup& \{\frac{\chi}{1-\xi}>c,\chi> 0,\xi< 1\},\nonumber 
\end{eqnarray}
where the sets of the union on the right hand side are separate. Furthermore 
\begin{equation*}
\lim_{c \to 0+} \mathbf{1}(\chi+\xi c>c)(\omega)=\mathbf{1}(\chi>0)(\omega)+\mathbf{1}(\chi=0,\xi>1)(\omega)
\end{equation*}
and
\begin{equation*}
\lim_{c \to \infty} \mathbf{1}(\chi+\xi c>c)(\omega)=\mathbf{1}(\xi>1)(\omega)+\mathbf{1}(\chi>0,\xi=1)(\omega)
\end{equation*}
for all $\omega \in \Omega$. \qed
\end{lemma}

\begin{lemma}\label{alleyksilemma} Assume $\I^1(A)<s<\I(A)$. 

Then we may choose $c=c_s$ and $n=n_s$ so that
\begin{equation}\label{osoitettava}
E((A_1\ldots A_n)^s\mathbf{1}(Y_n+A_1\ldots A_n c>c))>1.
\end{equation}
\end{lemma} 
\begin{proof}
By independence of financial risks
\begin{equation}\label{peruste1}
E((A_1\ldots A_k)^s)=E(A_1^s)\ldots E(A_k^s)\to \infty,
\end{equation}
as $k \to \infty$. On the other hand
\begin{equation}\label{peruste2}
E((A_1\ldots A_k)^s)=E((A_1\ldots A_k)^s\mathbf{1}(A_1\ldots A_k>1))+E((A_1\ldots A_k)^s\mathbf{1}(A_1\ldots A_k\leq1)),
\end{equation}
where 
\begin{equation*}
E((A_1\ldots A_k)^s\mathbf{1}(A_1\ldots A_k\leq1))\leq 1.
\end{equation*}
Since \eqref{peruste1} grows without limitation while the last term of \eqref{peruste2} remains uniformly bounded for each $k$, we see that there is a number $n \in \mathbb{N}$ for which
\begin{equation}\label{perusteko}
E((A_1\ldots A_n)^s\mathbf{1}(A_1\ldots A_n>1))>1.
\end{equation}

Choose $\chi=Y_n$ and $\xi=A_1\ldots A_n$ in Lemma \ref{laskulemma}. By representation \eqref{joukkoesitys}  and the positivity of the random variable $(A_1\ldots A_n)^s$ we get for each $c>0$ the estimate
\begin{eqnarray*}
& & E((A_1\ldots A_n)^s\mathbf{1}(Y_n+A_1\ldots A_n c>c)) \\
&\geq & E((A_1\ldots A_n)^s[\mathbf{1}(Y_n>0,A_1 \ldots A_n>1) \\
&+&\mathbf{1}(\frac{Y_n}{1-A_1 \ldots A_n}<c,Y_n\leq 0,A_1 \ldots A_n> 1)]).
\end{eqnarray*}
The previous minorant converges to $E((A_1\ldots A_n)^s\mathbf{1}(A_1 \ldots A_n>1))$, as $c \to \infty$ by the monotone convergence theorem. This limit exceeds the level $1$ by the result of formula \eqref{perusteko}.  
\end{proof}

We are now ready to present the proof. Firstly, if $\I(B)\leq \I^1(A)$, we may use the trivial bound 
\begin{equation}\label{trivialbound}
\limsup_{U_0 \to \infty} \frac{\log P(T_{U_0}<\infty)}{\log U_0}\geq \limsup_{U_0 \to \infty} \frac{\log P(Y_1>U_0)}{\log U_0} =-\I(B).
\end{equation}
Assume then that inequality $\I(B)>\I^1(A)$ holds. Iteration of the stochastic fixed point equation $\bar{Y}\stackrel{d}{=}B+A\bar{Y}^+$ combined with the monotonicity of the function $x \mapsto x^+$ yields
\begin{eqnarray*}
(\bar{Y})^+&\stackrel{d}{=}& (B_1+A_1\bar{Y}^+)^+ \\
&\geq & (B_1+A_1\bar{Y})^+ \\
&\stackrel{d}{=}& (B_1+A_1B_2+A_1 A_2 \bar{Y}^+)^+ \\
&\geq& (B_1+A_1B_2+A_1 A_2 \bar{Y})^+ \\
&\vdots& \\
&\stackrel{d}{=}&(Y_k+A_1\ldots A_k \bar{Y}^+)^+ \\
&\geq & (Y_k+A_1\ldots A_k \bar{Y})^+
\end{eqnarray*}
for each $k \in \mathbb{N}$. Here the random variable $\bar{Y}$ on the right hand side is independent of the random vectors $(A_1,B_1),\ldots ,(A_k,B_k)$.

Using the previous estimate we obtain for all $c>0$ and $k\in \mathbb{N}$ 
\begin{eqnarray*}
& &E((\bar{Y}-c)^s\mathbf{1}(\bar{Y}>c)) \\
& \geq & E((Y_k+A_1\ldots A_k \bar{Y}-c)^s\mathbf{1}(Y_k+A_1\ldots A_k \bar{Y}>c)) \\
&\geq &E((Y_k+A_1\ldots A_k \bar{Y}-c)^s\mathbf{1}(Y_k+A_1\ldots A_k c>c)\mathbf{1}(\bar{Y}>c)) \\
& \geq & E((A_1 \ldots A_k(\bar{Y}-c))^s\mathbf{1}(Y_k+A_1\ldots A_k c>c)\mathbf{1}(\bar{Y}>c)) \\
&=& E((A_1 \ldots A_k)^s\mathbf{1}(Y_k+A_1\ldots A_k c>c))E((\bar{Y}-c)^s\mathbf{1}(\bar{Y}>c)).
\end{eqnarray*}
Keeping in mind the assumption $\bar{y}=\infty$ this implies for any  $s<\I(\bar{Y})$ that
\begin{equation}\label{alleyksi}
E((A_1 \ldots A_k)^s\mathbf{1}(Y_k+A_1\ldots A_k c>c))\leq 1.
\end{equation}

Define, for given $c>0$ and $k \in \mathbb{N}$,
\begin{equation}\label{aaceekoo}
A_{c,k}=A_1 \ldots A_k \mathbf{1}(Y_k+A_1\ldots A_k c>c).
\end{equation}
Equation \eqref{alleyksi} implies the inequality
\begin{equation}\label{alleyksi2}
\I(\bar{Y})\leq \I^1(A_{c,k}).
\end{equation}

Assume $\epsilon>0$ is given. Choose $\hat{c}=\hat{c}_{\I^1(A)+\epsilon}$ and $\hat{n}=\hat{n}_{\I^1(A)+\epsilon}$ according to Lemma \ref{alleyksilemma}. Now 
\begin{equation*}
\I(\bar{Y})\leq \I^1(A_{\hat{c},\hat{n}})
\end{equation*}
by Equation \eqref{alleyksi2}. Since
\begin{equation*}
E((A_1\ldots A_{\hat{n}})^{\I^1(A)+\epsilon}\mathbf{1}(Y_n+A_1\ldots A_{\hat{n}} \hat{c}>\hat{c}))>1,
\end{equation*}
we deduce the inequality
\begin{equation*}
\I^1(A_{\hat{c},\hat{n}})\leq \I^1(A)+\epsilon,
\end{equation*}
which proves \eqref{ultimateruin} under assumptions of \eqref{assumptionsxx}. 

\subsubsection{General setting}

We will now omit assumptions of \eqref{assumptionsxx} and prove \eqref{ultimateruin}. If $\I(B)\leq\I^1(A)$, formula \eqref{ultimateruin} holds because the upper limit of Section \ref{upperbounds} is always valid and lower bound can be obtained from trivial estimate \eqref{trivialbound}.

We may now assume $\I(B)>\I^1(A)$. Fix small numbers $\epsilon_1>0$ and $\epsilon_2>0$. Then construct random variables $A_{\epsilon_1}$ and $B_{\epsilon_2}$ as in lemmas \ref{aarviolemma} and \ref{omamoma}. Now $A\geq A_{\epsilon_1}$ and $B\geq B_{\epsilon_2}$ almost surely. By \eqref{hassujuttu1} and the solution of the restricted case in \ref{restricted} we obtain
\begin{equation}\label{hassujuttu2}
\limsup_{U_0 \to \infty} \frac{\log P(T_{U_0}<\infty)}{\log U_0}\geq -\min(\I^1(A_{\epsilon_1}),\I(B_{\epsilon_2})).
\end{equation}
Since $\I^1(A_{\epsilon_1})=\I^1(A)+\epsilon_1$ and $\I(B_{\epsilon_2})=\I(B)+\epsilon_2$ and \eqref{hassujuttu2} holds for arbitrarily small numbers $\epsilon_1$ and $\epsilon_2$, we get 
\begin{equation*}
\limsup_{U_0 \to \infty} \frac{\log P(T_{U_0}<\infty)}{\log U_0}= -\min(\I^1(A),\I(B))
\end{equation*}
in all possible cases. \qed

\section{Conclusions and comments}\label{conclusions}

\subsection{Classical random walk} 

In Section \ref{intro} the classical random walk was briefly mentioned. Formally, a random walk can be recovered from model \eqref{modeldesc} by setting $A\equiv 1$. This leads to a process $(S_n)$, where
\begin{equation}\label{originalrandomwalk}
S_n=B_1+\ldots +B_n.
\end{equation}
Denote $\bar{S}=\sup_{n} S_n$. Model \eqref{modeldesc} with non-degenerate financial risks converges, whereas the process $S_n$ has a drift. For meaningful analysis this drift needs to be negative.

If $E(|B|)<\infty$, the negative drift is known to be equivalent with the condition $E(B)<0$. Under this assumption Borovkov, see \cite{borovkov1976stochastic} page 140 formula (54), derived a representation for the moment index $\I(\bar{S})$ using ascending ladder heights associated with the random process $(S_n)$. In this case 
\begin{equation*}
\I(\bar{S})=\I(B)-1,
\end{equation*}
which is completely different from the result obtained for the process $(Y_n)$.

\subsection{Final remarks}\label{sublund}

The quantity \eqref{lundbergdef} used in the proof of the main theorem has an alternative representation via moment indices.

Let $(Z_i)$ be an IID-sequence of positive random variables. We define 
\begin{equation}\label{lundberg1}
Z_\infty=\sum_{k=1}^\infty Z_1\ldots Z_{k}
\end{equation}
and
\begin{equation}\label{lundberg3}
\bar{Z}=\sup_{k}\{Z_1\ldots Z_k\}.
\end{equation}

The random variable defined in \eqref{lundberg1} makes sense when the stochastic series converges. The following lemma proves the equality of moment indices.

\begin{lemma} Let $Z$ be a positive random variable that satisfies assumptions \ref{oz1}-\ref{oz3} of Theorem \ref{viit}. Then the random variables $Z_\infty$ and $\bar{Z}$, defined by \eqref{lundberg1} and \eqref{lundberg3} respectively, satisfy
\begin{equation}\label{lundbergyht1}
\I^1(Z)=\I(Z_\infty)=\I(\bar{Z}).
\end{equation}
\end{lemma}
\begin{proof} To see that $\I^1(Z)=\I(Z_\infty)$, take $Z=A$ and $B \equiv 1$ in \eqref{ultimateruin}. 

For the last equality of \eqref{lundbergyht1}, note first that $\bar{Z}\leq Z_\infty$. By Lemma \ref{lin} part \ref{kohta3}, $\I(\bar{Z})\geq \I(Z_\infty)$. Furthermore, for any $k \in \mathbb{N}$ and $s>0$
\begin{equation}\label{lundbergyht2}
E(\bar{Z}^s)\geq E((Z_1 \ldots Z_k )^s)=E(Z^s)^k.
\end{equation}
Since formula \eqref{lundbergyht2} renders the inequality $\I^1(Z)<\I(\bar{Z})$ impossible, we conclude that $\I^1(Z)\geq \I(\bar{Z})$. 
\end{proof}

\section*{Acknowledgements}
This work was financially supported by the Finnish doctoral programme in stochastics and statistics (FDPSS) and by Academy of Finland, grant number 251170. Special thanks are due to Harri Nyrhinen for his diligent guidance throughout the writing process of the paper. In addition, suggestions made by an anonymous referee greatly improved the manuscript.

\bibliographystyle{acm}
\bibliography{C:/Users/JaakkoL/Desktop/MyBib/mybib}

\end{document}